\documentclass[reqno]{amsart}
\usepackage{enumerate}
\usepackage{amsfonts}
\usepackage{amsthm,amsmath}
\usepackage[integrals]{wasysym}
\usepackage{amssymb}
\usepackage[body={34cc,50cc},centering]{geometry}
\usepackage{graphics}

\newtheorem{theorem}{Theorem}[section]
 \newtheorem{corollary}[theorem]{Corollary}
 \newtheorem{lemma}[theorem]{Lemma}
 
 \theoremstyle{definition}
 
 \theoremstyle{definition}

 \theoremstyle{remark}

 \newcommand{\eps}{\varepsilon}

 \newcommand{\R}{\mathbb{R}}

 \newcommand{\abs}[1]{\left\vert#1\right\vert}

 \newcommand{\inner}[1]{\left(#1\right)}

 \newcommand{\reff}[1]{(\ref{#1})}

\begin{document}

\title[Particle paths in small rotational solitary waves]{Particle paths in small amplitude solitary waves with negative vorticity}

\author[L.-J. Wang]{Ling-Jun Wang}

\address{
College of Science, Wuhan University of Science and Technology, Wuhan 430081, China
}

\email{wanglingjun@wust.edu.cn}

\begin{abstract}
We
investigate the particle trajectories in solitary waves with vorticity,
where the vorticity is assumed to be negative and
decrease with depth. We show that the individual
 particle moves in a similar way as that in the irrotational case if
 the underlying laminar flow is favorable, that is, the flow is moving in the same direction as the wave propagation throughout the fluid,  and show that if the underlying current is not favorable, some particles in
 a sufficiently small solitary wave  move to the opposite direction of wave propagation along a path
 with a
 single loop or hump.
\end{abstract}

\maketitle

\section{Introduction}
The \textit{water-wave problem} concerns the gravity-driven flow of a
perfect fluid of unit density; the flow is bounded below by a rigid
horizontal bottom $\{Y=-d\}$ and above by a free surface
$\{Y=\eta(X,t)\}$, where $\eta$ depends upon the horizontal spatial
coordinate $X$ and time $t$. \textit{Steady waves} are waves which
propagate from left to right with constant speed $c$ and without
change of shape, so that $\eta(X,t)=\eta(X-ct)$.  \textit{Solitary
  waves} are steady waves which have the property that
$\eta(X-ct)\rightarrow 0$ as $X-ct\rightarrow \pm\infty$. We consider
in this paper the particle trajectories in a fluid as a solitary wave propagates on the free surface, assuming
that the flow admits a  negative
vorticity decreasing with depth.

There have been a series of works concentrating on the study of
solitary waves, in the setting of both irrotational flows
\cite{AT,BBB,CraigS, CE,CEH} etc.,
and  rotational flows which  become active only in the last few years.
One of the interests in the above works is the description of
individual
particle path in the fluid.  In \textit{irrotational} flow, particle
paths underneath a solitary wave are investigated
in \cite{Cqam,CE}
, both for the smooth solitary wave and the solitary wave of greatest
height for which the crest is a stagnation point. It was shown in \cite{CE} that in an irrotational solitary water wave, each particle is
transported in the wave direction but slower than the wave speed; as the solitary wave propagates,
all particles located ahead of the wave crest are lifted while
those behind have a downward motion, with the particle trajectory having asymptotically the same height above the flat bed. For \textit{rotational }flow, some results on particle paths under \textit{periodic} waves are obtained (cf. \cite{Ejde, Enonlinearity, Delia,WaJDE} for instance), following a series of works on the corresponding results in irrotational case (see \cite{Cinvention,CW2,HenryImrn} etc).

Recently, rigorous existence results on small-amplitude solitary
waves with arbitrary vorticity distribution were obtained in \cite{HurSolitary} and in
\cite{GW}, using generalized implicit function theorem of
Nash-Moser type and spatial dynamics method, respectively. The solitary waves established in \cite{HurSolitary,GW} are of elevation and decays exponentially to a horizontal laminar flow far up- and downstream. The study
of solitary waves of large-amplitude with an arbitrary distribution of vorticity remains open. So this arises the question of particle paths in a rotational small-amplitude solitary wave.
Following the pattern in \cite{CE} for irrotational case,  we
prove in this work the
corresponding results on
particle paths in rotational solitary waves  by using  the properties of solutions obtained in \cite{HurSolitary}. We consider in this work only solitary waves with negative vorticity, however with modifications our arguments can be applied also for positive vorticity.

Precisely we show that if the
vorticity is negative and increasing from bottom to surface of the
flow, and if the underlying current is favorable, i.e., is moving throughout the fluid in
the same direction as the wave,
then as time goes on the particle moves similarly as in the
irrotational case \cite{CE}.  We also show that  if the underlying current is not
favorable, then in solitary waves with
sufficiently small amplitude,  some particles above the flat bed move
to the opposite direction of wave propagation along a path with a single loop or a single hump.  Note that this single-loop kind of
path does not exist in the irrotational case.   In \cite{HurSolitary}, as in most of
the works on waves with vorticity,   the author
considered only waves that are not near breaking or
stagnation, i.e., the speed of an individual fluid particle is far
less than that of the wave itself throughout the fluid domain. We
consider only such waves as well. We do not consider the case of wave
with stagnation which is however studied  in \cite{CE} in the
irrotational case.

The paper is organized as follows. In Section \ref{sec p} we present
the mathematical formulation for the solitary waves and recall from
\cite{HurSolitary} some useful properties of the established solitary wave solutions.
Section \ref{sec v} contains some conclusions on the vertical and horizontal velocity that are relevant for
our purposes. The main result is presented and proved in Section
\ref{sec r}. In the final section we give two examples which lie in our settings.

\section{Preliminaries}\label{sec p}
We first describe the governing equations for rotational solitary water waves and then recall some properties available on their solitary wave solutions.
\subsection{The governing equations for rotational solitary water waves.}
Choose Cartesian coordinates $(X,Y)$ so that the horizontal $X$-axis is in the direction of wave propagation and the $Y$-axis points upwards. Consider steady waves traveling at constant speed $c>0$. In the frame of reference moving with the wave, which is equivalent to the change of variables $(X-ct,Y)\mapsto (x,y)$, we use
\[
  \Omega_{\eta}=\{(x,y)\in \R^2: -d<y<\eta(x)\},\quad 0 < d <\infty,
\]
 to denote the stationary fluid domain and
$(u(x,y),v(x,y))$ to denote denote the velocity field, and define the \textit{stream function} $\psi(x,y)$ by $\psi(x,\eta(x))=0$ and
\begin{eqnarray}\label{def psi}
  \psi_y=u-c,\quad \psi_x=-v.
\end{eqnarray}
Consider also only waves that are not near breaking or stagnation, so that $\psi_y(x,y)\leq -\delta<0$ in  $\bar{\Omega}_{\eta}$ for some $\delta>0$, which  implies that the vorticity $\omega=v_x-u_y$ is globally a function of the stream function $\psi$, denoted by $\gamma(\psi)$; see \cite{CW}.
The solitary-wave problem is then, for given $p_0<0$ and $\gamma\in C^1([0,-p_0]; \R)$, to find a real parameter $\lambda$, a domain $\Omega_{\eta}$ and a function $\psi\in C^2(\bar{\Omega}_\eta)$ such that
\begin{eqnarray}
  \psi_y<0,\quad (x,y)\in \bar{\Omega}_{\eta},\label{psi y negative}\\
  \triangle\psi=-\gamma(\psi),\quad (x,y)\in \Omega_{\eta},\label{equ psi}\\
  \abs{\nabla \psi}^2+2gy=\lambda,\quad y=\eta(x),\label{Bernouli}\\
  \psi=0,\quad y=\eta(x),\label{psi eta}\\
  \psi=-p_0, \quad y=-d,\label{psi -d}\\
  \eta(x)\rightarrow 0 \quad {\rm as} \abs{x}\rightarrow \infty,\label{eta to 0}\\
    \psi_x(x,y)\rightarrow 0 \quad {\rm as~~} \abs{x}\rightarrow \infty ~~{\rm uniformly ~~ for~~} y.\label{psi x to 0}
\end{eqnarray}
Here $g>0$ is the gravitational constant of acceleration,
\begin{eqnarray*}
  p_0=\int_{-d}^{\eta(x)}\psi_y(x,y)~dy<0
\end{eqnarray*}
is the relative mass flux (independent of $x$), and the boundary conditions \reff{eta to 0} and \reff{psi x to 0} express that the wave profile approaches a constant level of depth
and the flow is almost horizontal in the far field, respectively.
Moreover we require that the nontrivial solitary wave is of positive elevation, i.e., $\eta(x)>0$ for all $x\in\R$. It is therefore symmetric about its single crest and admits a strictly monotone wave profile on either side of this crest (see \cite {HurSymmetry}). Assuming the wave crest is located at $x=0$, the solitary-wave problem \reff{psi y negative}-\reff{psi x to 0} is thus supplemented with the symmetry and monotonicity conditions
\begin{eqnarray}\label{symmetry}
  \psi(-x,y)=\psi(x,y),\quad \eta(-x)=\eta(x),\quad {\rm and}~~ \eta'(x)<0~~{\rm for ~~}x>0.
\end{eqnarray}
We refer to \cite{HurSolitary, CW} for more details on the derivation of the system \reff{psi y negative}-\reff{psi x to 0}.

\subsection{Rotational solitary water waves.} We collect some properties of the solitary wave solutions established in \cite{HurSolitary}. Let
\begin{eqnarray}\label{def Gamma}
\Gamma(p)=\int_{0}^{p}\gamma(-s)ds \quad {\rm and} ~~\Gamma_{\min}=\min_{p\in[p_0,0]}\Gamma(p)\leq 0.
\end{eqnarray}
Given $p_0<0$ and $\gamma\in C^1([0,-p_0]; \R)$, for each $\lambda\in(-2\Gamma_{\min},\infty)$ the system \reff{psi y negative}-\reff{psi x to 0} admits a trivial solution  pair $(\eta(x),\Psi(y))$ defined on $\bar\Omega_0$, where  $\eta(x)\equiv0$, the stream function $\Psi(y)$ is $x$-independent and is the inverse of the function
\begin{eqnarray*}
  y(\Psi)=\int_{p_0}^{-\Psi}\frac{d p}{\sqrt{\lambda+2\Gamma(p)}}-d,
\end{eqnarray*}
and the fluid domain
\begin{eqnarray*}
  \Omega_0=\{(x,y)\in\R^2: -d<y<0\},\quad d=\int_{p_0}^{0}\frac{d p}{\sqrt{\lambda+2\Gamma(p)}}.
\end{eqnarray*}
The corresponding relative horizontal velocity and vertical velocity are thus given by
\begin{eqnarray}\label{U V}
  U(y)-c=\Psi_y(y)=-\sqrt{\lambda+2\Gamma(-\Psi(y))},\quad V(x,y)=-\Psi_x(y)\equiv 0.
\end{eqnarray}
Note that $\Psi(0)=0$ and $\Psi(-d)=-p_0$. Thus
\begin{eqnarray*}
  U(0)=c-\sqrt{\lambda}\quad {\rm and}\quad U(-d)=c-\sqrt{\lambda+2\Gamma(p_0)}.
\end{eqnarray*}
Throughout this paper, we adopt the terminology in \cite{CW2} to say this underlying trivial flow is \textit{favorable} if $U(y)\geq 0$ for all $y\in[-d,0]$, is \textit{adverse} if $U(y)<0$ for all $y\in[-d,0]$, and is \textit{mixed} if $U(y)$ changes sign. Note that favorable flow moves in the same direction as the wave propagation (i.e., to the right),  while adverse flow moves in the opposite direction of the wave propagation.

To ensure the existence of nontrivial small amplitude solitary waves, the parameter $\lambda$ must be chosen to satisfy $\lambda>\lambda_c$ but close to $\lambda_c$, where $\lambda_c>-2\Gamma_{\min}$ is the unique solution of
\begin{eqnarray}\label{def lambda c}
  \int_{p_0}^0\frac{dp}{\inner{\lambda_c+2\Gamma(p)}^{3/2}}=\frac{1}{g}.
\end{eqnarray}
For each such a given $\lambda$ and for given $p_0<0$, it was shown in \cite{HurSolitary}  that there exists a nontrivial small amplitude solitary-wave solution pair $(\eta(x),\psi(x,y))$ to \reff{psi y negative}-\reff{psi x to 0} defined on $\bar\Omega_\eta$, with $\eta(x)$ satisfying
\begin{eqnarray}\label{eta small}
  \abs{\eta(x)}+\abs{\eta'(x)}+\abs{\eta''(x)}\leq (\lambda-\lambda_c)r \quad {\rm for ~~all~~} x\in\R
\end{eqnarray}
and for some constant $r>0$ independent of $\lambda$, and the corresponding  horizontal velocity satisfying the following properties:
\begin{enumerate}[~~~~(P1)]
  \item  \label{u to 0} $u(x,\eta(x))\rightarrow c-\sqrt{\lambda}=U(0)$,  \quad {\rm as~~}$\abs{x}\rightarrow \infty $;
  \item \label{u to U} $u(x,y)\rightarrow U(y)$ \quad {\rm as~~}$\abs{x}\rightarrow \infty$ ~~{\rm uniformly~~for~~} $y$.
\end{enumerate}
The first property is obvious since by \reff{Bernouli}, \reff{eta to 0} and \reff{psi x to 0}, we have $(u(x,\eta(x))-c)^2=\psi^2_y(x,\eta(x))\rightarrow \lambda$ as $\abs{x}\rightarrow \infty$, which equivalently gives (P\ref{u to 0})
as we have assumed $u(x,y)-c=\psi_y(x,y)<0$ in $\bar\Omega_\eta$, while the
property (P\ref{u to U}) can be deduced from the construction of solitary wave solutions; see \cite{HurSolitary}. Indeed for $\lambda=\lambda_c+\eps$ with $\eps>0$, there exists a function $w^{\eps}(q,p)$, whose derivatives with respect to $(q,p)$ up to order 2 tend to 0 uniformly for $p$ as $\abs{q}\rightarrow \infty$, such that the horizontal velocity is determined by
\begin{eqnarray*}
  u(x,y)=c-\frac{1}{\inner{\lambda+2\Gamma(-\psi(x,y)}^{-1/2}+\eps w^{\eps}_p(\sqrt{\eps}x,-\psi(x,y))},
\end{eqnarray*}
where $w^{\eps}_p(q,p)$ denotes differentiation in
the $p$-variable.

 We conclude this section by recalling some properties of streamlines. Due to $\psi_y<0$ throughout $\bar\Omega_\eta$, we have that for all $p\in[-p_0,0]$ the streamline
\[
\{(x,y): \psi(x,y)=-p\}
\]
 is  a smooth curve $y=\sigma_p(x)$. Note that
 \begin{eqnarray}\label{sigma}
 \sigma_0(x)=\eta(x),\quad \sigma_{p_0}(x)=-d\quad {\rm and ~~}\sigma_p'(x)=-\frac{\psi_x(x,\sigma_p(x))}{\psi_y(x,\sigma_p(x))}.
 \end{eqnarray}
 Observing the fact that $\frac{1}{\abs{\psi_y}}=\frac{1}{c-u}$ is bounded and that $\psi_x(x,y)\rightarrow 0$ uniformly for $y$ as $\abs{x}\rightarrow \infty$, we deduce that for each fixed $y_0\in[0,\eta(0)]$ the streamline $y=\sigma_p(x)$ with $p=-\psi(0,y_0)$, passing through the point $(0,y_0)$, has an asymptote $y=l(y_0)$ as $\abs{x}\rightarrow\infty$, with $l(\eta(0))=0$ and $l(-d)=-d$.

\section{Vertical and horizontal velocity}\label{sec v}
We first divide the fluid domain $\Omega_\eta$ into two components
\[
 \Omega_-=\{(x,y)\in\R^2: x<0, 0<y<\eta(x)\} \quad {\rm and~~~} \Omega_+=\{(x,y)\in\R^2: x>0, 0<y<\eta(x)\},
\]
and denote their boundaries by
\[
 S_-=\{(x,y)\in\R^2: x<0, y=\eta(x)\}, \quad B_-=\{(x,y)\in\R^2: x<0, y=-d\},
\]
respectively
\[
 S_+=\{(x,y)\in\R^2: x>0, y=\eta(x)\}, \quad B_+=\{(x,y)\in\R^2: x>0, y=-d\}.
\]

\begin{lemma}\label{lem vertical}
  Suppose that $\gamma'(p)\leq 0$ for all $p\in[0,\abs{p_0}]$. Then
  \begin{enumerate}[(a)]
    \item \label{v=0} $v(x,-d)=0$ for all $x\in\R$, and $v(0,y)=0$ for $y\in [-d,\eta(0)]$ ;
    \item\label{lem vertical 2} $v(x,y)<0$ if $(x,y)\in \Omega_-\cup S_-$, and $v(x,y)>0$ if $(x,y)\in \Omega_+\cup S_+$;
    \item\label{lem vertical 3} $v_y(x,-d)<0$ if $x<0$, and $v_y(x,-d)>0$ if $x>0$;
    \item \label{lem vertical 4} $v_x(0,y)>0$ for $y\in(-d,\eta(0))$.
  \end{enumerate}
\end{lemma}
\begin{proof}
  Since $v=-\psi_x$, the first result follows from \reff{psi -d} and the symmetry property $\psi(-x,y)=\psi(x,y)$ in \reff{symmetry}.

  Next we only prove the lemma for $x>0$ and the results for $x<0$ can be proved similarly. Differentiating \reff{psi eta} with respect to $x$ gives $v=-\psi_x=\psi_y\eta'(x)$, from which follows $v>0$ for $(x,y)\in S_+$ as $\psi_y<0$ and $\eta'(x)<0$ for $x>0$ in view of \reff{psi y negative} and \reff{symmetry}. To prove $v>0$ in $\Omega_+$, we assume first on the contrary that there exists a point $(x_0,y_0)\in\Omega_+$ such that $v(x_0,y_0)=-\eps<0$. Then we can find a bounded domain
  \[
    \Omega_{+,k}=\{(x,y)\in\R: 0<x<k,-d<y<\eta(x)\},\quad k\in\R^+,
  \]
  such that $(x_0,y_0)\in\Omega_{+,k}$. Moreover in view of \reff{psi x to 0} we can choose $k$ sufficiently large so that $v(k,y)> -\eps$ for $y\in[-d,\eta(k)]$. This means that $v$ attains its minimum  at the interior point $(x_0,y_0)$ of $ \Omega_{+,k}$, which contradicts to the strong maximum principle applied to $v$ on the domain $\bar\Omega_{+,k}$, as $v$ satisfies
  $\triangle v+\gamma'(\psi)v=0$ with $\gamma'(p)\leq 0$ by differentiating \reff{equ psi}. Therefore we have $v\geq 0$ in $\Omega_+$. If $v=0$ at a point $(x_0,y_0)$ of $\Omega_+$. Again we can choose a bounded domain $\Omega_{+,\hat k}$ containing $(x_0,y_0)$. Then $v\geq0$ on the boundary of $\Omega_{+,\hat k}$. We can thus apply the strong maximum principle on $\Omega_{+,\hat k}$ once more to conclude that $v\equiv0$ on $\bar\Omega_{+,\hat k}$, which contradicts $v>0$ on the half surface $S_+$. This proves $v>0$ in $\Omega_+$.

   Since $v$ attains its minimum in $\bar\Omega_+$ on the half bed $B_+$ and on the crest line $x=0$,  Hopf's maximum principle ensures that $v_y>0$ on $B_+$ and $v_x>0$ on $\{(0,y); -d<y<\eta(0)\}$, completing the proof.
\end{proof}

The above lemma combined with \reff{sigma} and the fact that $\psi_y<0$ gives

\begin{corollary}\label{cor streamline}
  The streamline $y=\sigma_p(x)$ with $p\in(p_0,0]$ satisfies that $\sigma_p'(x)>0$ for $x<0$ and $\sigma_p'(x)<0$ for $x>0$.
\end{corollary}

\begin{lemma}\label{lem horizontal}
  Suppose that $\gamma(p),\gamma'(p)\leq0$ for all $p\in [0,\abs{p_0}]$, that $0<\lambda-\lambda_c<\eps$ with $\eps$ small such that nontrivial solitary wave exists, and that $c\geq\sqrt{\lambda+2\Gamma(p_0)}$. Then $u(x,y)>0$ for $x\in \bar\Omega_\eta$.
\end{lemma}
\begin{proof}
  We assume throughout this proof that $\gamma(p)\not\equiv 0$ for $p\in [0,\abs{p_0}]$, since we already know that for $\gamma(p)\equiv 0$, i.e., the irrotational case, $u(x,y)>0$ in $\bar \Omega_\eta$; see \cite{CE}. Recall that $U(y)$ is the horizontal velocity of the trivial laminar flow. Thus $U'(y)=\Psi_{yy}(x,y)=-\gamma(\Psi(y))\geq0$, and there exists an interval $I$ such that the strict inequality holds for $y\in I$ since $\gamma(p)\not\equiv 0$ by assumption. This combined with \reff{U V} and the assumption $c\geq\sqrt{\lambda+2\Gamma(p_0)}$ gives
  \begin{eqnarray*}
    U(y)\geq U(-d)=c-\sqrt{\lambda+2\Gamma(p_0)}\geq 0~~{\rm for} ~~y\in[-d,0],~~~{\rm and }~~~~U(0)>0.
  \end{eqnarray*}

  It was proved in \cite{EV} that if $\gamma(p)\leq0$ for all $p\in [0,\abs{p_0}]$, then
  \begin{eqnarray}\label{u mono s}
    \frac{d}{dx}u(x,\eta(x))\geq 0~~~~{\rm for~~} x<0,~~~~{\rm and} ~~~~\frac{d}{dx}u(x,\eta(x))\leq 0~~{\rm for~~} x>0.
  \end{eqnarray}
  In other words, along the free surface $u$ increases from $x=-\infty$ to the crest $x=0$, and thereafter it is decreasing. On the bottom $B=\{(x,y)\in\R^2: y=-d\}$, we have, in view of $u_x=-v_y=\psi_{xy}$ and Lemma \ref{lem vertical}-\reff{lem vertical 3},
  \begin{eqnarray}\label{u mono b}
    u_x(x,-d)>0~~~~{\rm for~~} x<0,~~~~{\rm and} ~~~~ u_x(x,-d)<0~~~~{\rm for~~} x>0.
  \end{eqnarray}
  Then the fact that $U(-d)\geq 0$ and $U(0)>0$ together with the monotonicity properties \reff{u mono s} and \reff{u mono b} yields $u>0$ on the free surface and on the bottom. Differentiating \reff{equ psi} with respect to $y$ yields that
  $u$ satisfies
  \[
    \triangle u+\gamma'(\psi)u=\gamma'(\psi) c\leq 0.
  \]
  Finally considering $U(y)\geq 0$ for $y\in[-d,0]$ and the properties (P\ref{u to 0})-(P\ref{u to U}), we can deduce $u>0$ in $\bar\Omega_\eta$ by using the strong maximum principle as in the proof of Lemma \ref{lem vertical}-\reff{lem vertical 2}, completing the proof.
\end{proof}

The above lemma considers the case $c\geq\sqrt{\lambda+2\Gamma(p_0)}$, that is the underlying flow is favorable. For the remaining cases, we can only derive some conclusions on the horizontal velocity $u$ for some special vorticity function $\gamma$  and for $\lambda$  sufficiently close to the corresponding $\lambda_c$ defined in \reff{def lambda c}. To be exact, we prove
\begin{lemma}\label{lem horizontal small}
  Let $\lambda_c$ be determined in \reff{def lambda c} by the vorticity function $\gamma(p)$ with $p\in[0,\abs{p_0}]$.
  \begin{enumerate}[(a)]
    \item\label{ux} If $\gamma'(p),\gamma''(p)\leq0$ for $p\in[0,\abs{p_0}]$ and $\gamma(0)\sqrt{\lambda_c}>-g$, then there exists  $\eps_1>0$ such that for $\lambda\in(\lambda_c,\lambda_c+\eps_1)$, we have
      \begin{eqnarray*}
       u_{x}> 0~~ {\rm for~~} (x,y)\in\Omega_-, ~~~{\rm and~~} u_{x}<0~~ {\rm for~~} (x,y)\in\Omega_+;
      \end{eqnarray*}
    \item\label{uy} If $\gamma(p)<0$ for all $p\in[0,\abs{p_0}]$, then there exists  $\eps_2>0$ such that for $\lambda\in(\lambda_c,\lambda_c+\eps_2)$,
    we have
    \begin{eqnarray*}
       u_{y}> 0~~ {\rm for~~} (x,y)\in\Omega.
      \end{eqnarray*}
  \end{enumerate}
 \end{lemma}
\begin{proof}
  $\reff{ux}$ Similar results have been obtained for periodic waves in \cite{Enonlinearity}, so we will adapt the proof there to our present solitary wave case.
  We prove only the result for  $(x,y)\in\Omega_+$ and the rest can be proved similarly. Note it suffices to prove that the stream function $\psi$ satisfies
  \begin{eqnarray*}
    \psi_{xy}<0~~ {\rm for~~} (x,y)\in\Omega_+.
  \end{eqnarray*}

  Since $\psi(x,\eta(x))=0$, we have $\psi_x=-\eta'\psi_y$ on the surface, which combined with \reff{Bernouli} gives
  \[
    \psi_y(x,\eta(x))=-\sqrt{(\lambda-2g\eta)/(1+\eta'^2)}.
  \]
  Differentiating the above equality and $\psi_x=-\eta'\psi_y$ along the surface gives respectively
  \begin{eqnarray*}
    \psi_{xy}+\eta'\psi_{yy}=-\partial_x\sqrt{\frac{\lambda-2g\eta}{1+\eta'^2}}\quad {\rm and}\quad\psi_{xx}+\eta'\psi_{yy}+2\eta'\psi_{xy}+\eta''\psi_y=0,~~~{\rm on ~~}y=\eta(x).
  \end{eqnarray*}
  Moreover on the surface one has
  \[
    \psi_{xx}+\psi_{yy}=-\gamma(0).
  \]
  Combination of the above three equalities gives
  \[
    \psi_{xy}(x,\eta(x))=\frac {\eta'\inner{g(1-\eta'^4)+2\eta''(\lambda-2g\eta)+\gamma(0)\sqrt{\lambda-2g\eta}(1+\eta'^2)^{3/2}}}
    {\sqrt{\lambda-2g\eta}(1+\eta'^2)^{5/2}}.
  \]
  In view of \reff{eta small} and $\gamma(0)\sqrt{\lambda_c}>-g$, we have when $\lambda$ is sufficiently close to $\lambda_c$ that
  \[
     g(1-\eta'^4)+2\eta''(\lambda-2g\eta)+\gamma(0)\sqrt{\lambda-2g\eta}(1+\eta'^2)^{3/2}>0.
  \]
   This gives $\psi_{xy}<0$ on the surface since $\eta'<0$ for $x>0$ and $\sqrt{\lambda-2g\eta}(1+\eta'^2)^{5/2}>0$. Since $v(0,y)=0$, we have $\psi_{xy}=-v_y=0$ on the line $x=0$. On the bottom $\psi_{xy}<0$ holds due to Lemma \ref{lem vertical}-\reff{lem vertical 3}. And $\psi_{xy}=u_x\rightarrow 0$ as $x\rightarrow\infty$ since $u(x,y)\rightarrow U(y)$ as $x\rightarrow\infty$. Finally $\psi_{xy}$ satisfies
  \begin{eqnarray*}
    (\triangle+\gamma')\psi_{xy}=-\gamma''\psi_x\psi_y\geq 0.
  \end{eqnarray*}
  The conclusion follows from the maximum principle.

  $\reff{uy}$ Recall from \cite{HurSolitary} that for $\lambda=\lambda_c+\eps$, there exists a function $w^\eps(q,p)$ such that $u(x,y)$ is determined by
  \begin{eqnarray}\label{expression u}
  u(x,y)=c-\frac{1}{(\lambda+2\Gamma(-\psi(x,y)))^{-1/2}+\eps w^{\eps}_p(\sqrt{\eps}x,-\psi(x,y))},
\end{eqnarray}
 where $w^{\eps}_p(q,p)$ stands for the differentiation of the function $w^{\eps}(q,p)$ with respect to $p$, and the family $\{w^{\eps}(q,p); 0\leq \eps<1\}$ satisfies that  for $\eps$ small
 \begin{eqnarray}\label{property w}
   \sup\Bigl\{\abs{\partial^j_q\partial^k_p w^{\eps}(q,p)};j+k\leq2, (q,p)
   \in\R\times[p_0,0] \Bigr\}\leq r,
 \end{eqnarray}
 with $r>0$ some constant independent of $\eps$. Differentiating \reff{expression u} with respect to $y$ gives
 \begin{eqnarray*}
   u_y=\frac{-\psi_y\inner{-\gamma(\psi)\inner{\lambda+2\Gamma(-\psi)}^{-3/2}+\eps w^\eps_{pp}(\sqrt{\eps}x,-\psi)}}
   {\inner{\inner{\lambda+2\Gamma(-\psi)}^{-1/2}+\eps w^\eps_p(\sqrt{\eps}x,-\psi)}^2}.
 \end{eqnarray*}
 Since $\gamma(p)<0$, if we denote $\gamma_{\max}=\max_{p\in[0,\abs{p_0}]}\gamma(p)$, then $\gamma_{\max}<0$ and $\Gamma_{\max}=\max_{p\in[p_0,0]}\Gamma(p)>0$. As a result, observing
 \begin{eqnarray*}
   -\gamma(\psi)\inner{\lambda+2\Gamma(-\psi)}^{-3/2}+\eps w^\eps_{pp}(\sqrt{\eps}x,-\psi)\geq-\gamma_{\max}\inner{\lambda+2\Gamma_{\max}}^{-3/2}+\eps w^\eps_{pp}(\sqrt{\eps}x,-\psi)
 \end{eqnarray*}
 and
 \begin{eqnarray*}
   \lim_{\eps\rightarrow 0}\inner{-\gamma_{\max}\inner{\lambda+2\Gamma_{\max}}^{-3/2}+\eps w^\eps_{pp}(\sqrt{\eps}x,-\psi)}
   =-\gamma_{\max}\inner{\lambda_c+2\Gamma_{\max}}^{-3/2}
 \end{eqnarray*}
 due to \reff{property w},
 we get, in view of $-\gamma_{\max}\inner{\lambda_c+2\Gamma_{\max}}^{-3/2}>0$, that
 \begin{eqnarray*}
   -\gamma(\psi)\inner{\lambda+2\Gamma(-\psi)}^{-3/2}+\eps w^\eps_{pp}(\sqrt{\eps}x,-\psi)>0
 \end{eqnarray*}
 when $\eps$ is sufficiently small. This combined with $-\psi_y>0$ gives $u_y>0$ for $\eps$ sufficiently small.
\end{proof}

\begin{lemma}\label{u mono streamline}
  If $\gamma(p)<0,\gamma'(p),\gamma''(p)\leq0$ for $p\in[0,\abs{p_0}]$ and $\gamma(0)\sqrt{\lambda_c}>-g$, then  for $\lambda\in(\lambda_c,\lambda_c+\eps_0)$ with $\eps_0=\min\{\eps_1,\eps_2\}$, along every streamline $y=\sigma_p(x)$ with $p\in(p_0,0)$, the horizontal velocity $u$ is strictly decreasing in $\Omega_+$ and strictly increasing in $\Omega_-$ as a function of $x$.
\end{lemma}
\begin{proof}
   Since
   \begin{eqnarray*}
     \frac{d }{dx}u(x,\sigma_p(x))=u_x+u_y\sigma_p'(x)=u_x+u_y\frac{v}{u-c}
   \end{eqnarray*}
   due to \reff{sigma}, the conclusion follows immediately from Lemma \ref{lem vertical} and Lemma \ref{lem horizontal small}.
\end{proof}

\section{Main result}\label{sec r}
Recalling that $y=l(y_0)$ is the streamline asymptote introduced at the end of Section \ref{sec p}, we are now ready to state our main result describing the particle
trajectories in a solitary wave with negative vorticity.

\begin{theorem}\label{thm}
  Let the flux $p_0<0$, the vorticity function $\gamma\in C^2(0,\abs{p_0})$
 be given, and let $\lambda_c$ be determined as in \reff{def lambda c}. Assume that $\gamma(p)<0,\gamma'(p)\leq0$ for all $p\in [0,\abs{p_0}]$, and that $\lambda>\lambda_c$ such that nontrivial solitary wave solutions to \reff{psi y negative}-\reff{psi x to 0} exist.  Then the following results hold.
  \begin{enumerate}[(a)]
  \item\label{under crest} Any particle above the bed reaches at some instant $t_0$ the location $(X_0,Y_0)$ below the wave crest $(X_0,\eta(0))$;
    \item\label{Tra positive} For $c\geq\sqrt{\lambda+2\Gamma(p_0)}$, as time $t$ runs on $(-\infty,t_0)$, the particle above the flat bed moves to the right and upwards, while for $t>t_0$ the particle moves to the right and downwards, as in Figure \ref{fig1}-(a);  the particle on the flat bed moves in a straight line to the right at a positive speed;
    \item\label{Tra nonpositive} If assume additionally that $\gamma(0)\sqrt{\lambda_c}>-g$ and $\gamma''(p)\leq 0$ for all $p\in [0,\abs{p_0}]$, then there exists $\eps_0>0$ such that for $\lambda\in(\lambda_c,\lambda_c+\eps_0)$ and
    \begin{enumerate}[(i)]
      \item\label{Tra mixed} for $\sqrt{\lambda}<c<\sqrt{\lambda+2\Gamma(p_0)}$, there does not exist a single pattern for all the  particles above the flat bed; if $u(x,-d)\geq 0$, depending on the relation between the asymptote $y=l(Y_0)$ and the zero point $y_*$ of $U(y)$, some particles move to the right along a path with a single hump as described in \reff{Tra positive}, some move along a single-loop path to the left, as in Figure \ref{fig1}-(b); if $u(x,-d)<0$, there are additionally some particles moving to the left along a single-hump path; see Figure \ref{fig1}-(c);
      \item\label{Tra negative} for $c\leq\sqrt{\lambda}$, depending on the signs of $u(0,\eta(0))$ and $u(0,-d)$, there are three possibilities for the particles above the flat bed; see Figure \ref{fig3};
      \item\label{Tra bottom} for a particle on the flat bed in both the cases (i) and (ii), if $u(x,-d)\leq0$ it moves to the left in a straight line, while if $u(x,-d)>0$  it firstly has a backward-forward pattern of motion and then moves to the left in a straight line;
    \end{enumerate}
    \item\label{Tra infinity} The particle trajectory is strictly above the asymptote $Y=l(Y_0)$ of the streamline $Y=\sigma_p(X-ct)$ with $p=-\psi(0,Y_0)$.
  \end{enumerate}
  \end{theorem}

\begin{proof}
  The path (past and future) $(X(t),Y(t))$ of a particle with location $(X(0),Y(0))$ at time $t=0$ is given by the solution of the non-autonomous system
  \begin{eqnarray*}
  \left\{
    \begin{array}{ll}
      \dot X=u(X-ct,Y),\\
      \dot Y=v(X-ct,Y).
    \end{array}
    \right.
  \end{eqnarray*}
  Working in the moving frame $x=X-ct$ and $y=Y$, we transform the above system into 
  \begin{eqnarray}\label{equ x y}
    \left\{
    \begin{array}{ll}
      \dot x=u(x,y)-c,\\
      \dot y=v(x,y).
    \end{array}
    \right.
  \end{eqnarray}
  This is a  Hamiltonian system with Hamiltonian $\psi(x,y)$ in view of \reff{def psi}, meaning that the solutions $(x(t),y(t))$ of \reff{equ x y} lie on the streamlines. All solutions of \reff{equ x y} are defined globally in time since the boundedness of the right-hand side prevents blow-up in finite time.

  \reff{under crest} Since in the moving frame the wave crest is assumed to be located at $x=0$ in our setting, the sign of $x(t)$ describes the position of the particle with respect to the wave crest at time $t$: the particle is exactly below the crest if $x(t)=0$, is ahead of the crest if $x(t)>0$ while is behind the crest if $x(t)<0$.  Note that $\dot x=u(x,y)-c\leq -\delta< 0$ throughout $\bar\Omega_{\eta}$. This uniform upper bound on $\dot x$ implies that $x(t)\rightarrow \mp\infty$ as $t\rightarrow \pm\infty$ and there is a unique time $t_0$ such that $x(t_0)=0$. That is, to each fluid particle moving within the water there corresponds a unique time $t_0\in\R$ so that at $t=t_0$ the particle is exactly below the wave crest, while afterwards it is located behind the wave crest, the wave crest being behind the particle for $t<t_0$.

  In the subsequent of the proof, we always assume that the particle is located below the crest at time $t=t_0$ at the location $(X_0,Y_0)=(X(t_0),Y(t_0))$, or equivalently $(x(t_0),y(t_0))=(0,Y_0)$ which implies in fact $X_0=ct_0$.

    \reff{Tra positive}
     Assume $c\geq\sqrt{\lambda+2\Gamma(p_0)}$.  For a particle located above the flat bed, we prove  only  the statement for $t>t_0$ and the case $t<t_0$ can be proved similarly. Since $x(t)$ is strictly decreasing, one has $X(t)-ct=x(t)<0$ for $t>t_0$, which combined with Lemma \ref{lem vertical}-\reff{lem vertical 2} implies $\dot Y<0$ for $t>t_0$. Furthermore we have $\dot X>0$ for all time by Lemma \ref{lem horizontal}. That means the particle moves to the right and downwards as time runs on $(t_0,+\infty)$. The statement for particles on the bed follows from Lemma \ref{lem vertical}-\reff{v=0}
     and Lemma \ref{lem horizontal}. The particle path for this case is depicted in Figure \ref{fig1}-(a).

 \begin{figure}
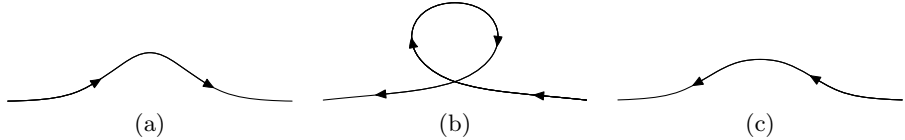

\begin{center}
 \includegraphics{picture.1}\quad
 \includegraphics{picture.2}\quad
\includegraphics{picture.3}
\end{center}
\caption{\small {Particle path with a single: (a) hump to the right; (b) loop to the left; (c) hump to the left.}}\label{fig1}
\end{figure}

\begin{figure}
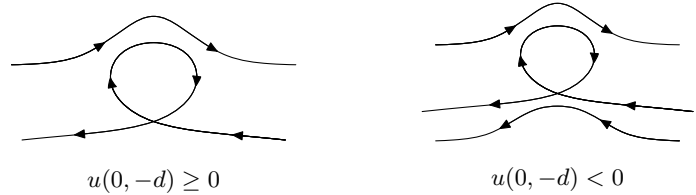

\begin{center}
  \includegraphics{picture.4}\hspace{1.5cm}
   \includegraphics{picture.5}
\end{center}
\caption{\small {Particle path above the flat bed in a small solitary wave with a mixed underlying current.}}\label{fig2}
\end{figure}

     \reff{Tra nonpositive}-(i)  Since
     $\sqrt{\lambda}<c<\sqrt{\lambda+2\Gamma(p_0)}$, we have $U(0)>0$
     and $U(-d)<0$, thus $U(y)$ has  a unique zero point, say at
     $y_*$.
Assume first that $u(0,-d)\geq 0$. The monotonicity properties of $u$ provided by Lemma \ref{lem horizontal small}-\reff{ux} and \reff{u mono b} show that in $\bar\Omega_+$ the level set $\{u=0\}$ consists of a continuous curve ${\mathcal C}_+$ in the moving frame. The curve  is confined between $y=-d$ and $y=y_*$, and can be parameterized by $x=h(y)$ with $h'(y)>0$, $h(y)\rightarrow +\infty$ as $y\rightarrow y_*$, and $h(-d)\geq 0$ with the equality holding when $u(x,-d)=0$. In $\bar\Omega_-$ the level set $\{u=0\}$ is given by the reflection ${\mathcal C}_-$ of the curve ${\mathcal C}_+$ across the line $x=0$. Below the curve ${\mathcal C}_+$ and ${\mathcal C}_-$ we have $u<0$ (including the bottom), while above and between the two curves we have $u>0$. Furthermore, in view of Lemma \ref{u mono streamline}, if a streamline and the curve ${\mathcal C}_+$ (rep. ${\mathcal C}_-$) intersect, they intersect exactly once.

     Recall that $(x(t),y(t))$ lies on the  streamline $y=\sigma_p(x)$
     with $p=-\psi(0,Y_0)$ since $(x(t_0),y(t_0))=(0,Y_0)$, and recall
     that $y=l(Y_0)$ is the asymptote of the streamline passing through the point $(0,Y_0)$.
 In virtue of Corollary \ref{cor streamline}, the path $(x(t),y(t))$ is located below $y=Y_0$ and above the asymptote $y=l(Y_0)$ for all time $t$.  If $l(Y_0)\geq y_*$, then we have from the above arguments that $u>0$ for all time $t$, while $v<0$ if $t>t_0$ and $v>0$ if $t<t_0$. This is the same situation as described in \reff{Tra positive}, so that the particle moves to the right along a path with a single hump, as depicted in Figure \ref{fig1}-(a).

     If $l(Y_0)<y_*$, the particle trajectories above the flat bed are as shown in Figure \ref{fig1}-(b). Indeed, in this case, as time $t$ increases from $-\infty$, the path $(x(t),y(t))$ intersects successively the curve ${\mathcal C}_+$ from below at $t=t_+$, the vertical line $x=0$ from right at $t=t_0$, and the curve ${\mathcal C}_-$ from above at $t=t_-$.  In the time interval $t\in(-\infty,t_+)\cup(t_-,+\infty)$ we know that $u<0$ so that in the physical variables $(X,Y)$ the particle $(X(t),Y(t))$ moves to the left. In the time interval $t\in(t_+,t_-)$ we have $u>0$ so that $(X(t),Y(t))$ moves to the right. Also we have $v>0$ when $t<t_0$ so that $(X(t),Y(t))$ moves up, while when $t>t_0$ we have $v<0$ so that $(X(t),Y(t))$ moves down. Thus in this case the particle above the flat bed moves to the left along a path with a single loop.

    It remains to treat the case $u(0,-d)<0$. Note $u(0,\eta(0))>0$ in virtue of $U(0)>0$ and \reff{u mono s}. Thus, in view of Lemma
    \ref{lem horizontal small}-\reff{uy}, there exists a unique
    $\tilde y\in(-d, \eta(0))$ such that $u(0,\tilde y)=0$. Observe
    $\tilde y<y_*$ since $u_x<0$ in $\Omega_+$ and $u(x,y)\rightarrow U(y)$ as $x\rightarrow \infty$. The corresponding
    curve ${\mathcal C}_+$ is now located between $y=y_*$ and
    $y=\tilde y$, intersecting the line $x=0$ at $y=\tilde y$. For
    $Y_0>\tilde y$, the paths $(X(t),Y(t))$ are similar as encountered
    in the case when $u(0,-d)\geq 0$. While for $Y_0\leq \tilde y$, we
    have $(x(t),y(t))$ is below ${\mathcal C}_+$ and ${\mathcal C}_-$,
    so that $u<0$ for all the time. As a result, the particles move to
    the left along a single-hump path; see Figure \ref{fig1}-(c).

    We depict all the possible trajectories in this case in Figure \ref{fig2}.

\reff{Tra nonpositive}-(ii) This case can be treated similarly as that in the
above, so we omit the details and show the particle trajectories in
Figure \ref{fig3}.

\begin{figure}
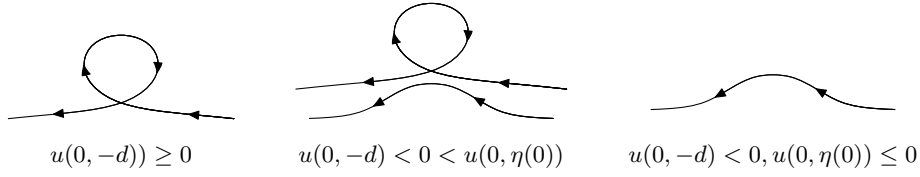

\begin{center}
  \includegraphics{picture.6}\quad\quad
 \includegraphics{picture.7}\quad \quad
  \includegraphics{picture.8}
\end{center}
\caption{\small {Particle path above the flat bed in a small solitary wave with an adverse underlying current.}}\label{fig3}
\end{figure}

 \reff{Tra nonpositive}-(iii) The conclusions for particles on the flat bed
 follow from Lemma \ref{lem vertical}-\reff{v=0} and \reff{u mono b}.

  \reff{Tra infinity} Observing $\dot Y=v(X-ct,Y)$, $X(t)-ct=x(t)\rightarrow -\infty$ as $t\rightarrow+\infty$ and $v(x,y)$ converges to $0$ as $\abs{x}\rightarrow \infty$ uniformly in $y$, we have the existence of some $\alpha$ such that $\lim_{t\rightarrow+\infty}Y(t)=\alpha$. Since $\psi(x,y)$ is the Hamiltonian function of the system \reff{equ x y}, we have
  \[
    \psi(x(t),y(t))=\psi(x(t_0),y(t_0))=\psi(0,Y_0).
  \]
  Thus $y(t)=\sigma_p(x(t))$ with $p=-\psi(0,Y_0)$. Consequently
  \[
    \lim_{t\rightarrow +\infty}Y(t)=\lim_{t\rightarrow +\infty}y(t)=\lim_{t\rightarrow +\infty}\sigma_p(x(t))=l(Y_0).
  \]
  The proof is completed.
\end{proof}

\noindent{\bf Remark}
  For the path with a single loop, if we define the size of the loop by the net horizontal distance moved by the particle between the two instants $t_+$ and $t_-$ when its horizontal velocity changes sign, that is,
  \begin{eqnarray*}
    X(t_-)-X(t_+),
  \end{eqnarray*}
  then the size decreases with depth. In fact  it can be computed by
  \begin{eqnarray*}
    X(t_-)-X(t_+)=\int_{t_+}^{t_-}\frac{d X}{dt}dt=\int_{t_+}^{t_-}u\inner{x,\sigma_p(x)}dt.
  \end{eqnarray*}
  Since $u_y>0$ and $\frac{d \sigma_p(x)}{dp}=-\frac{1}{\psi_y}>0$, we have $X(t_-)-X(t_+)$ decreases as $p$ decreases.

\section{Examples}\label{sec e} In this final section we give two examples of the vorticity function which satisfy the conditions imposed in our main theorem.
\subsection{Negative constant vorticity} In the case of negative constant vorticity $\gamma(p)=-\omega_0$ for $p\in[0,\abs{p_0}]$ with $\omega_0>0$, we have obviously $\gamma'(p),\gamma''(p)\leq 0$. So it remains to verify $\gamma(0)\sqrt{\lambda_c}>-g$, or equivalently, $\lambda_c<g^2/\omega_0^2$. Recall from \reff{def lambda c} that
$\lambda_c>-2\Gamma_{\min}$ is such that
\begin{eqnarray*}
  \int_{p_0}^0\frac{dp}{\inner{\lambda_c+2\Gamma(p)}^{3/2}}=\frac{1}{g}
\end{eqnarray*}
holds, where $\Gamma(p)=-\omega_0 p$ and $\Gamma_{\min}=0$ in this case by \reff{def Gamma}. Set
\[
  F(\lambda)=\int_{p_0}^0\frac{dp}{{\inner{\lambda-2\omega_0 p}}^{3/2}}.
\]
 Then direct computation shows that $F'(\lambda)<0$, $F(\lambda)\rightarrow +\infty$ as $\lambda\rightarrow 0+$, and
 \begin{eqnarray*}
   \lim_{\lambda\rightarrow g^2/\omega_0^2}F(\lambda)=\frac{1}{g}-\frac{1} {\sqrt{g^2-2\omega_0^3p_0}}<\frac{1}{g}.
 \end{eqnarray*}
 Thus there exists a unique $\lambda_c\in(0,g^2/\omega_0^2)$ such that $F(\lambda_c)=1/g$.

 \subsection{Negative affine linear vorticity} For the vorticity $\gamma(p)=-a p+b$ with $a>0$ and $b<0$,
 we have obviously $\gamma(p)<0$, $\gamma'(p),\gamma''(p)\leq 0$ for $p\in[0,\abs{p_0}]$ and
 \[
  F(\lambda)=\int_{p_0}^0\frac{dp}{{\inner{\lambda+2\Gamma(p)}}^{3/2}}=\int_{p_0}^0\frac{dp}{{\inner{\lambda+ap^2+2b p}}^{3/2}},
\]
with $\lambda>-2\Gamma_{\min}=0$. Then we may verify as above that $F(\lambda)$ is strictly decreasing with $\lim_{\lambda\rightarrow 0}F(\lambda)=+\infty$ and $\lim_{\lambda\rightarrow g^2/b^2}F(\lambda)<0$. This gives the existence of $\lambda_c\in(0,g^2/b^2)$ such that $F(\lambda_c)=1/g$ and consequently $\gamma(0)\sqrt{\lambda_c}>-g$.

\subsection*{Acknowledgements}  This work was supported  in part by
 Foundation of WUST(2010xz019) and by NSF grants of China 10901126.

\end{document}